\documentclass{amsart}
\usepackage{amsfonts,amssymb}
\usepackage{enumerate,graphicx}
\usepackage{caption}
\usepackage{subcaption}
\usepackage[square,numbers]{natbib}
\usepackage{mathrsfs,bbm}
\usepackage{tikz,tikzscale,tikz-cd}
\usepackage{multirow,xparse,fp}
\usepackage{environ}
\usepackage{amstext}
\usepackage{hyperref}
\usepackage{diagbox}
\usepackage[pagewise]{lineno}%\linenumbers
\usepackage{comment}
\newtheorem{theorem}{Theorem}[section]

\newtheorem{corollary}[theorem]{Corollary}
\newtheorem{lemma}[theorem]{Lemma}
\theoremstyle{definition}

\newtheorem{definition}[theorem]{Definition}

\newtheorem{problem}{Problem}

\numberwithin{equation}{section}%讓第二章節第一個方程式變成編號(2.1)
\allowdisplaybreaks[4]%align可跨頁

\begin{document}
	\title{On the independence of shifts defined on $\mathbb{N}^d$ and trees}
	
	\author[Jung-Chao Ban]{Jung-Chao Ban}
	\address[Jung-Chao Ban]{Department of Mathematical Sciences, National Chengchi University, Taipei 11605, Taiwan, ROC.}
	\address{Math. Division, National Center for Theoretical Science, National Taiwan University, Taipei 10617, Taiwan. ROC.}
	\email{jcban@nccu.edu.tw}
	
	\author[Guan-Yu Lai]{Guan-Yu Lai}
	\address[Guan-Yu Lai]{Department of Mathematical Sciences, National Chengchi University, Taipei 11605, Taiwan, ROC.}
	\email{gylai@nccu.edu.tw}

	
	%\thanks{}
	
	%\date{}
	
	%\baselineskip=1.1\baselineskip
	
	% -------------------------------------------------------------
	\begin{abstract}
		In this paper, we study the independence of shifts defined on $\mathbb{N}^d$ ($\mathbb{N}^d$ shift) and trees (tree-shift). Firstly, for the completeness of the article, we provide a proof that an $\mathbb{N}^d$ shift has positive (topological) entropy if and only if it has an independence set with positive upper density. Secondly, we obtain that when the base shift $X$ is a hereditary shift, then the associated tree-shift $\mathcal{T}_X$ on an unexpandable tree has positive entropy if and only if it has an independence set with positive density. However, the independence of the tree-shift on an expandable tree differs from that of $\mathbb{N}^d$ shifts or tree-shifts on unexpandable trees. The boundary independence property is introduced and we prove that it is equivalent to the positive entropy of a tree-shift on an expandable tree.   
	\end{abstract}
	\maketitle

\section{Introduction}

Let $\mathcal{A}$ be a finite set with $\left\vert\mathcal{A}\right\vert\geq 2$, where $|\cdot|$ is the cardinality of a set. Let $X\subseteq \mathcal{A}^{\mathbb{N}}$ be an $\mathbb{N}$ shift, and
the shift $\sigma :\mathcal{A}^{\mathbb{N}}\rightarrow \mathcal{A}^{\mathbb{N}}$ is a continuous transformation given by $\sigma(x)=(x_{i+1})_{i=1}^{\infty }$ if $x=(x_{i})_{i=1}^{\infty }\in \mathcal{A}^{\mathbb{N}}$. The symbolic system $(X,\sigma)$ forms a dynamical system and the \emph{entropy}, say $h(X)$ \cite{LM-1995}, of $\left( X,\sigma \right)$ is defined as 
\begin{equation}
h(X)=\lim_{n\rightarrow \infty }\frac{\log \left\vert B_{n}(X)\right\vert }{n}\text{,}  \label{6}
\end{equation}
where $B_{n}(X)$ is the set of all $n$-blocks of $X$. The limit (\ref{6}) exists by using the subadditive argument \cite{LM-1995}. It is evident that the entropy of the dynamical system $(X,\sigma )$ is a significant invariant and receives a great deal of attention. It provides the information about
the complexity, chaos phenomena, and classification theory of the system $(X,\sigma )$. For a more detailed discussion of the research topics mentioned above, we refer the reader to \cite{alseda2000combinatorial, downarowicz2011entropy,kitchens2012symbolic, LM-1995,viana2016foundations}.

\subsection{Independence set for $\mathbb{N}^{d}$ shifts}
In recent studies, the `independence set' concept has been employed to characterize the positive entropy of a symbolic dynamical system $(X,\sigma)$ as an indicator of the system's randomness (cf. \cite{falniowski2015two,glasner1995quasi,huang2006local, kerr2007independence, weiss2000single}). A set $J\subseteq 
\mathbb{N}$ is called an \emph{independence set} for an $\mathbb{N}$ shift $X$ if for every function $\eta :J\rightarrow \mathcal{A}$ there is a point $x=\{x_{j}\}_{j=1}^{\infty }\in X$ such that $x_{i}=\eta (i)$ for every $i\in J$. A set $A\subseteq \mathbb{N}$ has \emph{asymptotic density} $\alpha $ if
the limit 
\[
d(A)=\lim_{n\rightarrow \infty }\frac{\left\vert A\cap [1,n]\right\vert }{n} 
\]
exists and is equal to $\alpha $, where $[m,n]:=\{m,...,n\}$ ($\forall m\leq n$). For the relation of topological entropy with the $\mathbb{N}$ shift, B. Weiss \cite{weiss2000single} prove the following result.

\begin{theorem}[Theorem 8.1\cite{weiss2000single}]
\label{Thm: 2}If $X\subseteq [0,1]^{\mathbb{N}}$ is a binary $\mathbb{N}$ shift, then $h(X)$ is positive if and only if $X$ is independent over a set whose asymptotic density exists, and is positive. Furthermore, if $X\subseteq [0,r-1]^{\mathbb{N}}$, then $h(X)>\log (r-1)$ if and only if $X$ is independent over a set $A$ whose asymptotic density exists, and is positive.
\end{theorem}

Later, Falniowski et al. prove the strengthened result that Theorem \ref{Thm: 2} is also true for Shnirelman density\footnote{The definition of Shnirelman density of $A\subseteq \mathbb{N}$ is replaced $\lim$ in the definition of $d(A)$ by $\inf$.} \cite[Theorem 1 \& Theorem 2]{falniowski2015two}. In accordance with above theorems, a positive entropy shift $X$ ensures a large number of nodes in $\mathbb{N}$ where $X$ behaves randomly. The following question arises naturally.

\begin{problem}
Dose Theorem \ref{Thm: 2} hold true for multidimensional shifts?
\end{problem}

Due to the many different types of multidimensional shifts, our focus will be on the following two types of shifts: (1). the $\mathbb{N}^{d}$ shifts with $d\geq 2$, and (2). the tree-shift, i.e., the shift defined on the free semigroup with $d$ generators (the conventional $d$-tree) or Markov-Cayley trees with $d$ generators. The shifts of type (1) and (2) exhibit distinct behavior because the corresponding underlined spaces, i.e., the $\mathbb{N}^{d}$ and free semigroup are quite different. We refer the reader to \cite{ceccherini2010cellular} for a detailed discussion of shifts defined on an amenable group (e.g., $\mathbb{N}^{d}$) and to \cite{kerr2016ergodic} for those defined on the sofic groups (e.g., free semigroup with $d$ generators). Our main results are as follows.

\textbf{1}. Theorem \ref{Thm: 2} is extended to $\mathbb{N}^{d}$ shifts (Theorem \ref{Thm: 3}). We emphasize that our result on an $\mathbb{N}^{d}$ shift (Theorem \ref{Thm: 3}) is comparable to Theorem \ref{Thm: 2} except the asymptotic density in Theorem \ref{Thm: 2} is replaced by the upper density in Theorem \ref{Thm: 3}. While this result has been previously obtained by D. Kerr and H. Li \cite[Section 12]{kerr2016ergodic} for amenable group actions, we provide a proof using extended techniques introduced in \cite{falniowski2015two} for the convenience of the reader.  

\textbf{2}. The concept of an independence set for tree-shifts is being studied. We separate the set of trees into two categories, namely, the unexpandable ($\gamma =1$) and expandable trees ($\gamma \geq 1$). According to Theorem \ref{Thm: 1} (2), the tree-shift on a $\gamma =1$ tree has the same result as $\mathbb{N}$ shift (Theorem \ref{Thm: 2}). However, as we discuss below, the independence set for the tree-shifts on a $\gamma >1$ tree is quite different from that of $\mathbb{N}^{d}$ shifts. Theorem \ref{Thm: 4} presents an example that reveals that the independence set concept is no longer valid for a tree-shift on $\gamma >1$ trees. A new concept of the independence set (boundary independence property) for those tree-shifts is provided, and
Theorem \ref{Thm: 5} shows that such a property characterizes the positive entropy of a tree-shift. The boundary independence property is also applied to characterize positive surface entropy of a tree-shift (Theorem \ref{Thm: 6}). Consequently, Corollary \ref{Cor 1} establishes the equivalence of positivity between entropy and surface entropy.

First, we present our results for the independence of the $\mathbb{N}^d$ shifts below. Let $2\leq d\in \mathbb{N}$, the upper density of a set of nodes in $\mathbb{N}^{d}$ is introduced, which is used to study the independence set for $\mathbb{N}^{d}$ shifts. Let $X\subseteq \mathcal{A}^{\mathbb{N}^{d}}$ be an $\mathbb{N}^{d}$
shift. A subset $S\subseteq \mathbb{N}^{d}$ is called an \emph{independence set} of $X$ if for any $u\in \mathcal{A}^{S}$, there exists an $x=(x_{\mathbf{i}})_{\mathbf{i}\in \mathbb{N}^{d}}\in X$ such that $x_{\mathbf{i}}=u_{\mathbf{i}}$ for all $\mathbf{i}\in S$. The \emph{upper density} of a set $S\subseteq \mathbb{N}^{d}$ is defined by 
\begin{equation}
\overline{d}(S)=\limsup\limits_{n_{1},\ldots ,n_{d}\rightarrow \infty }\frac{\left\vert S\cap \prod_{i=1}^{d}[1,n_{i}]\right\vert }{\left\vert
\prod_{i=1}^{d}[1,n_{i}]\right\vert }\text{.}  \label{4}
\end{equation}

According to Theorem \ref{Thm: 3} below, the positivity of $h(X)$ indicates that $X$ behaves randomly on many nodes (with positive upper density) of $\mathbb{N}^{d}$.

\begin{theorem}[$\mathbb{N}^{d}$ shifts]
\label{Thm: 3}Let $X$ be an $\mathbb{N}^{d}$ shift over $\mathcal{A}=[0,1]$. Then, $h(X)>0$ if and only if $X$ is independent over a set whose upper density is positive.
\end{theorem}

\subsection{Tree-shifts and independence sets}
The second aim of this article is to study the independence of tree-shifts. A tree-shift is a type of multidimensional shift and has received much attention in past two decades (\cite{aubrun2012tree, ban2017mixing, ban2017tree, PS-2017complexity,petersen2020entropy}). Before presenting our main result, we first provide the background of the tree-shift. $T$ is called a \emph{tree} if $T$ is a countable graph that is locally finite without loops and with a root $\epsilon$, and we define it as a Cayley graph with $d$ generators $\{g_{1},\ldots ,g_{d}\}$. Let $M\in
[0,1]^{d\times d}$, the associated \emph{Markov-Cayley tree }$T^{M}$ is defined as 
\[
T^{M}=\{\epsilon\}\cup\{g_1,...,g_d\}\cup \{g_{i_{1}}\cdots g_{i_{n}}:n\geq 2\text{, }M(g_{i_{j}},g_{i_{j+1}})=1\text{ }\forall 1\leq j \leq n-1\}\text{.} 
\]

Let $T$ be a tree, we denote by $T_{i}$ the set of vertices in $T$ with length $i$ for $i\in \mathbb{N}\cup \{0\}$, where the length of $g\in T$ is the number of edges from $\epsilon $ to $g$. The union of $T_{i}$ from $i=0$ to $i=n$ is denoted by $\Delta _{n}:=\cup _{i=0}^{n}T_{i}$. The $n$-\emph{block} is a function $u:\Delta_{n}\rightarrow \mathcal{A}$ and a \emph{tree-shift} is a set $\mathcal{T}\subseteq \mathcal{A}^{T}$ of labeled trees (where a labeled tree is a function $t:T\rightarrow \mathcal{A}$.) which avoid all of a certain set of forbidden blocks. A tree-shift $\mathcal{T}$ is called a 
\emph{tree-shift of finite type} (\emph{tree-SFT}) if the cardinality of the forbidden set is finite. Let $X\subseteq \mathcal{A}^{\mathbb{N}}$ be an $\mathbb{N}$ shift, we confine ourselves to a wider class of tree-shifts $\mathcal{T}_{X}$, namely, the \emph{tree-shift associated with} $X$ , which is
defined as 
\begin{equation}
\mathcal{T}_{X}=\left\{x\in \mathcal{A}^{T}:(x_{g_{i_{1}}g_{i_{2}}\cdots
g_{i_{j}}})_{j\in \mathbb{N}\cup\{0\}}\in X\text{ for any }
(g_{i_{1}}g_{i_{2}}\cdots g_{i_{j}})_{j\in \mathbb{N}\cup\{0\}}\subseteq T\right\}\text{,}
\label{2}
\end{equation}
where $g_{i_0}:=\epsilon$. The $\mathbb{N}$ shift $X$ is called the \emph{base shift} of the $\mathcal{T}_{X}$ \cite{ban2024topological}. The concept of $\mathcal{T}_{X}$ was introduced by Petersen and Salama \cite{petersen2020entropy} and it indicates that the language of $X$ is the source of every word in each ray of the tree $T$. The concept of the tree-shift is the same as the $\mathbb{N}^{d}$ vertex shift if $X$ is a one-step SFT and $T=\mathbb{N}^{d}$ \cite{LM-1995}. It should also be mentioned that $\mathcal{T}_{X}$ is a type of $G$-shift (where $G$ represents a group.) \cite{
fiorenzi2009periodic, frisch2017symbolic,lemp2017shift} if $T$ is replaced by a group $G$.

Let $\mathcal{T}\subseteq \mathcal{A}^{T}$ be a tree-shift, we denote by $\mathcal{P}(\Delta_{n},\mathcal{T})$ the \emph{canonical projection} of $\mathcal{T}$ into the subblock on $\Delta_{n}$. That is, $\mathcal{P}(\Delta_{n},\mathcal{T})=\{(x_{g})_{g\in \Delta_{n}}\in \mathcal{A}^{\Delta_{n}}:x\in \mathcal{T}\}$. 

The \emph{entropy} of $\mathcal{T}$ is defined as 
\begin{equation}
h(\mathcal{T})=\limsup\limits_{n\rightarrow \infty }\frac{\log \left\vert \mathcal{P}(\Delta_{n},\mathcal{T})\right\vert }{\left\vert \Delta_{n}\right\vert }\text{.}  \label{3}
\end{equation}
The limsup in (\ref{3}) is indeed a limit for a tree-shift on a conventional $d$-tree with $d\in \mathbb{N}$ \cite{PS-2017complexity, petersen2020entropy}. The same result applies to a large class of tree-shifts on Markov-Cayley trees \cite{ban2022stem}.

Similar to the independent set for an $\mathbb{N}^{d}$ shift, we provide the definition of the independence set for a tree-shift. A subset $S\subseteq T$ is called an \emph{independence set }of a tree-shift $\mathcal{T}$ if $\forall u\in \mathcal{A}^{S}$, there exists an $t=(t_{g})_{g\in T}\in\mathcal{A}^{T}$ such that $t_{g}=u_{g}$ $\forall g\in S$. The structure of $T$ is crucial to the investigation of the independence set for the $\mathcal{T}_{X}$. The concept of `expanding numbers', defined below, is used to classify all trees. For any $T$, the following number $\gamma_{T}$ is called the \emph{expanding number} of a tree $T$ whenever the following limit exists. 
\begin{equation}
\gamma _{T}:=\lim_{n\rightarrow \infty }\frac{\left\vert T_{n+1}\right\vert 
}{\left\vert T_{n}\right\vert }\geq 1\text{.}  \label{1}
\end{equation}
The number (\ref{1}) measures how the cardinality of $T_{n}$ spreads as $n$ tends to the infinity. A tree $T$ is called \emph{expandable} if $\gamma _{T}>1$ (e.g., the conventional $2$ tree), and \emph{unexpandable} if $\gamma _{T}=1$ (e.g., the Markov-Cayley tree $T^{M}$ with $M=\left[ 
\begin{array}{rr}
1 & 1 \\ 
0 & 1
\end{array}
\right] $). Later on, we will demonstrate that the independence set phenomenon behaves differently on shift spaces defined on both types of trees. The \emph{upper density} of a set $S\subseteq T$ is defined by 
\begin{equation}
\overline{d}(S)=\limsup\limits_{n\rightarrow \infty }\frac{\left\vert S\cap\Delta_{n}\right\vert }{\left\vert \Delta_{n}\right\vert }\text{.}\label{7}
\end{equation}
If the limsup in (\ref{7}) is indeed a limit, then we call it the \emph{limit} of $S\subseteq T$ and write $d(S)$.   

\begin{theorem}[Tree-shift on $\protect\gamma =1$ trees]
\label{Thm: 1}Let $X$ be a hereditary shift and $T$ be an unexpandable Markov-Cayley tree. Then,
\begin{enumerate}
\item the entropies of $\mathcal{T}_{X}$ and $X$ are coincident.

\item the tree-shift $\mathcal{T}_{X}$ has positive entropy if and only if $\mathcal{T}_{X}$ has an independence set with positive density.
\end{enumerate}

Furthermore, if $X$ is an $\mathbb{N}$ shift with zero entropy, and $T$ is an unexpandable tree, then $\mathcal{T}_{X}$ contains no independence set with positive entropy.
\end{theorem}

It is revealed by Theorem \ref{Thm: 1} that the phenomenon of independence sets for shifts on unexpandable trees is comparable to that of $\mathbb{N}$ shifts (Theorem \ref{Thm: 2}) or $\mathbb{N}^{d}$ shifts (Theorem \ref{Thm: 3}) for $d\geq 2$. In addition, Ban et al. \cite{ban2024topological} demonstrate that 
\begin{equation}
h(\mathcal{T}_{X})=h(X)  \label{5}
\end{equation}
for $X$ is a coded system with finite generators or it satisfies the almost specification property. The result of (\ref{5}) is extended to a large class of shifts (i.e., the hereditary shift) on unexpandable trees in Theorem \ref{Thm: 1} (1). Many classical shifts, such as subshifts of finite type, sofic shifts, and beta shifts, are included in the class of hereditary shifts.
However, our proof of Theorem \ref{Thm: 1} is not applicable to all shifts. Therefore, the question mentioned below remains unanswered.

\begin{problem}
Does Theorem \ref{Thm: 1} hold true for all shifts?
\end{problem}

There is a dramatic difference in the theme of the independence set for tree-shifts on $\gamma >1$ trees. Theorem \ref{Thm: 4} below shows that the classical definition of the independence set is not appropriate for shifts on $\gamma >1$ trees.

\begin{theorem}[No independence set for a tree-shift on $\protect\gamma >1$ trees]\label{Thm: 4}
Let $T$ be an expandable tree. Then, there exists a tree-shift $\mathcal{T}\subseteq [0,1]^{T}$ which has positive entropy but it contains no independence set with positive density.
\end{theorem}

To illustrate the positive entropy for shifts on $\gamma >1$ trees, we introduce a new concept of the boundary independence property below.

\begin{definition}[Boundary independence property]
Let $\mathcal{T}\subseteq \mathcal{A}^{T}$ be a tree-shift. We call a tree-shift $\mathcal{T}\subseteq \mathcal{A}^{T}$ satisfies the \emph{boundary independence property} if there exist $l\geq 1$ and a sequence $\{S_{n}\}_{n=l}^{\infty }\subseteq T$ such that (1). $S_{n}\subseteq \Delta_{n}\setminus \Delta_{n-l}$ for all $n\geq l$, (2). each $S_{n}$ is an independence set for $\mathcal{T}$, and (3). $\limsup\limits_{n\rightarrow\infty }\frac{\left\vert S_{n}\cap \Delta_{n}\right\vert }{\left\vert\Delta_{n}\right\vert }>0$.
\end{definition}

Theorem \ref{Thm: 5} demonstrates that the boundary independence property characterizes the positive entropy of a tree-shift on a $\gamma >1$ tree.

\begin{theorem}[Tree-shifts on $\protect\gamma >1$ trees]\label{Thm: 5}
Let $\mathcal{T}$ be a tree-shift on $T$ over $[0,1]$. Then, the following assertions are equivalent.
\begin{enumerate}
\item the tree-shift $\mathcal{T}$ has positive entropy,

\item the tree-shift $\mathcal{T}$ has the boundary independence property.
\end{enumerate}

Furthermore, if $\mathcal{T}$ is a tree-shift on $T$ over $[0,r-1]$ with $r\geq 2$, then $h(\mathcal{T})>\log (r-1)$ if and only if $\mathcal{T}$ has the boundary independence property.
\end{theorem}

 The \emph{topological surface entropy}, say $h^{(S)}(\mathcal{T})$, is introduced in \cite{ban2022analogue} and defined as 
\[
h^{(S)}(\mathcal{T})=\limsup\limits_{n\rightarrow \infty }\frac{\log\left\vert B_{n}^{(S)}(X)\right\vert }{\left\vert T_{n}\right\vert }\text{,} 
\]
where $B_{n}^{(S)}(\mathcal{T})=\{t|_{T_{n}}:t\in \mathcal{T}\}$ is the set of blocks on $T_{n}$. The quantity of $h^{(S)}(\mathcal{T})$ for a tree-shift $\mathcal{T}$ is used to study the topological sequence entropy of the tree-shift. It has been demonstrated that the results of the topological sequence entropy for tree-shifts on $\gamma >1$ trees \cite{ban2022analogue} are different than those of $\mathbb{N}$ shifts (cf. \cite{huang2009combinatorial,kamae2002sequence, tan2009set}). Theorem \ref{Thm: 6} demonstrates that the boundary independence property characterizes the positive surface entropy of a tree-shift on a $\gamma >1$ tree. 

\begin{theorem}[Tree-shifts on $\protect\gamma >1$ trees]\label{Thm: 6}
Let $\mathcal{T}$ be a tree-shift on $T$ over $[0,1]$. Then, the following assertions are equivalent.
\begin{enumerate}
\item the tree-shift $\mathcal{T}$ has positive surface entropy,

\item the tree-shift $\mathcal{T}$ has the boundary independence property.
\end{enumerate}
\end{theorem}

Combining Theorems \ref{Thm: 5} and \ref{Thm: 6}, we have the following result.
\begin{corollary}\label{Cor 1}
    Let $T$ be an expandable tree and let $\mathcal{T}$ be a tree-shift on $T$. Then, $h(\mathcal{T})>0$ if and only if $h^{(S)}(\mathcal{T})>0$. 
\end{corollary}

In the remainder of this article, we provide the complete proof for Theorem \ref{Thm: 3} in Section \ref{sec2}, Theorem \ref{Thm: 1} in Section \ref{sec3}, and Theorems \ref{Thm: 4}, \ref{Thm: 5} and \ref{Thm: 6} in Section \ref{sec4}.

\section{Proof of Theorem \ref{Thm: 3}}\label{sec2}
Before proving Theorem \ref{Thm: 3}, some necessary notations, terminologies, and useful lemmas are provided. For the sake of simplicity, we only consider the case when $d=2$, and the case when $d>2$ can be treated in the same manner. For $a\in\mathcal{A}$ and $x\in \mathcal{A}^{\mathbb{N}^2}$, 
\begin{equation*}
   \chi_a(x) =\left\{ (i,j)\in\mathbb{N}^2: x_{i,j}=a \right\}.
\end{equation*}
Let $X\subseteq \mathcal{A}^{\mathbb{N}^2}$ be an $\mathbb{N}^2$ shift and let $B_{k_1,k_2}(X)$ be the set of all $k_1\times k_2$ blocks appearing in $X$. For $w\in B_{k_1,k_2}(X)$, 
\begin{equation*}
    \|w\|_a:=\left|\left\{(i,j)\in [1,k_1]\times [1,k_2]: w_{i,j}=a\right\}\right|.
\end{equation*}
Denote
\begin{equation}\label{eqM}
    M_{k_1,k_2}^a(X):=\max \left\{\|w\|_a:w\in B_{k_1,k_2}(X)\right\},
\end{equation}
and let $\overline{w}^{(k_1,k_2)}\in B_{k_1,k_2}(X)$ be a block which attains the maximal number of occurrences of the symbol $a$. That is,
\begin{equation}\label{eq wbar}
    \left\|\overline{w}^{(k_1,k_2)}\right\|_a= M_{k_1,k_2}^a(X).
\end{equation}

Note that for $n_1,n_2,m_1,m_2\geq 1$,
\begin{equation*}
    M_{n_1+n_2,m_1}^a(X)\leq M_{n_1,m_1}^a(X)+M_{n_2,m_1}^a(X),
\end{equation*}
and
\begin{equation*}
    M_{n_1,m_1+m_2}^a(X)\leq M_{n_1,m_1}^a(X)+M_{n_1,m_2}^a(X).
\end{equation*}
By Fekete’s Lemma, we have 
\begin{equation}\label{eq Fr}
    \lim_{k_1,k_2\to\infty}\frac{M_{k_1,k_2}^a(X)}{k_1\cdot k_2}=\inf_{k_1,k_2\geq 1}\frac{M_{k_1,k_2}^a(X)}{k_1\cdot k_2}:=\mbox{Fr}_a(X).
\end{equation}

The following lemma provide a variation of \cite[Theorem 3, Lemmas 1, 2, 3 and 4]{falniowski2015two}.

\begin{lemma}\label{lemma1} The following assertions hold true.
  \begin{enumerate}
      \item  For any $k_1,k_2>1$, there exists a block $u\in B_{k_1,k_2}(X)$ such that for all $1\leq j\leq k_1$
    \begin{equation}\label{lma1-0}
        \left\|u|_{[1,j]\times[1,k_2]}\right\|_a\geq  j\cdot k_2\cdot\text{Fr}_a(X),
    \end{equation}
where $u|_{[s_1,s_2]\times [t_1,t_2]}:=\left\{u_{s,t}:  s_1\leq s\leq s_2,~ t_1\leq t\leq t_2\right\}$.   
\item For any $a\in \mathcal{A}$, there exists a block $w_a\in X$ such that 
    \begin{equation*}
        \overline{d}(\chi_a(w_a))=\mbox{Fr}_a(X).
    \end{equation*}
    \item Let $0<\epsilon \leq \frac{1}{2}$ and $n_1,n_2\geq 1$. Then, 
    \begin{equation*}
        \sum_{j=0}^{\left\lfloor \epsilon n_1\cdot n_2\right\rfloor}\binom{n_1\cdot n_2}{j}\leq 2^{n_1\cdot n_2\cdot H(\epsilon)},
    \end{equation*}
    where $H(\epsilon):=-\epsilon \log \epsilon -(1-\epsilon)\log (1-\epsilon)$.
    \item Let $\mathcal{F}$ be a family of binary blocks of size $n\geq 0$. Then $| \mathbb{I}(\mathcal{F})|\geq |\mathcal{F}|$, where $\mathbb{I}(\mathcal{F})$ is the collection of all independence sets of $\mathcal{F}$.
    \item Let $\mathcal{F}\subseteq [0,1]^n$ be a family of binary blocks of size $n\geq 1$. If for some $1 \leq  k \leq  n$ we have
    \begin{equation*}
         \left|\mathcal{F}\right|>\sum_{j=0}^{k-1}\binom{n}{j},
    \end{equation*}
    then $\mathcal{F}$ is independent over some set with cardinality $k$.
    \item Let $X$ be a binary $\mathbb{N}^2$ shift with positive topological entropy. Then there is an $\epsilon> 0$ such that for every $k_1,k_ 2\geq 1$ there is a set $J\subseteq [1,k_1]\times [1,k_2]$ with $\lfloor \epsilon k_1k_2\rfloor$ elements which is an independence set for $X$.
  \end{enumerate}
\end{lemma}

\begin{proof}
\item[\bf (1)] Arguing contrapositively, there exist $k_1,k_2>1$ such that for any block $u\in B_{k_1,k_2}(X)$, there exists $1\leq j\leq k_1$ such that
    \begin{equation}\label{pflma1-1}
         \frac{\left\|u|_{[1,j]\times [1,k_2]}\right\|_a}{j\cdot k_2}< \text{Fr}_a(X)~(\leq 1).
    \end{equation}
    Note that $ \frac{\left\|u|_{[1,j]\times[1,k_2]}\right\|_a}{j\cdot k_2}=\frac{b}{c}$ with $0\leq b<c\leq k_1\cdot k_2$ and $k_1,k_2$ are fixed, we have that
    \begin{equation}\label{pflma1-3}
        \alpha_0:=\min \left\{\mbox{Fr}_a(X)-\frac{b}{c}>0: 0\leq b<c\leq k_1\cdot k_2\right\}
    \end{equation}
    is well-defined and is positive. We now take a positive integer $N$ such that $\left\lfloor N\alpha_0 \right\rfloor=1$. By (\ref{pflma1-1}) and $(p_1,p_2)=(k_1^2N+1,k_2)$, we can divide $\overline{w}^{(p_1,p_2)}$ (defined in (\ref{eq wbar})) into at least
$Nk+1$ pieces, with each being of size $k_1\times k_2$ at most, and for all but at most one of them (\ref{pflma1-1}) fails. We now choose $1\leq j_1,j_2,...,j_t\leq k_1$ iteratively from (\ref{pflma1-1}) such that $p_1-\sum_{i=1}^t j_i <k_1$, and for each $1\leq \ell \leq t$,
\begin{equation}\label{pflma1-2}
     \frac{\left\|\overline{u}^{(p_1,p_2)}|_{\left[\sum_{i=1}^{\ell-1}j_i+1,\sum_{i=1}^{\ell}j_i\right]\times [1,k_2]}\right\|_a}{j_\ell\cdot k_2}<\mbox{Fr}_a(X).
\end{equation}
Then, by (\ref{pflma1-3}), 
\begin{equation}\label{pflma1-4}
    \alpha_\ell:= \mbox{Fr}_a(X)-\frac{\left\|\overline{u}^{(p_1,p_2)}|_{\left[\sum_{i=1}^{\ell-1}j_i+1,\sum_{i=1}^{\ell}j_i\right]\times[1,k_2]}\right\|_a}{j_\ell\cdot k_2}\geq \alpha_0 ~(\forall 1\leq \ell \leq t).
\end{equation}
Thus, by (\ref{pflma1-4}), we have
\begin{equation}\label{pflma1-5}
  \begin{aligned}
    \left\|\overline{u}^{(p_1,p_2)}|_{\left[\sum_{i=1}^{\ell-1}j_i+1,\sum_{i=1}^{\ell}j_i\right]\times[1,k_2]}\right\|_a
    =&j_\ell\cdot k_2\cdot \left(\mbox{Fr}_a(X)-\alpha_{\ell}\right)\\
    \leq &j_\ell\cdot k_2\cdot \left(\mbox{Fr}_a(X)-\alpha_0\right) ~(\forall 1\leq \ell \leq t).
\end{aligned}  
\end{equation}
Hence, by (\ref{pflma1-5}),
\begin{align*}
\left\|\overline{u}^{(p_1,p_2)}\right\|_a=&\sum_{\ell=1}^t\left\|\overline{u}^{(p_1,p_2)}|_{\left[\sum_{i=1}^{\ell-1}j_i+1,\sum_{i=1}^{\ell}j_i\right]\times [1,k_2]}\right\|_a\\
&+\left\|\overline{u}^{(p_1,p_2)}|_{\left[\sum_{i=1}^{\ell}j_i+1,p_1\right]\times[1,k_2]}\right\|_a\\
    \leq &p_1k_2\left(\mbox{Fr}_a(X)-\alpha_0\right)+k_1 k_2\\
    =&p_1p_2\mbox{Fr}_a(X)-(k_1^2N+1)\alpha_0 k_2+k_1k_2\\
    \leq &p_1p_2\mbox{Fr}_a(X)-k_1^2N\alpha_0 k_2+k_1k_2\\
    \leq &p_1p_2\mbox{Fr}_a(X)-k_1^2k_2+k_1k_2<p_1p_2\mbox{Fr}_a(X),
\end{align*}
which contradicts with definitions (\ref{eq wbar}) and (\ref{eq Fr}).
\item[\bf (2)] By (1), we have a sequence of blocks $\{u^{(k,k)}\}_{k=1}^\infty$, with $u^{(k,k)}\in B_{k,k}(X)$, satisfying (\ref{lma1-0}). It can be obtained easily by using a standard compactness argument, then we have a subsequence $\{k_i\}_{i=1}^\infty$ of $\{k\}_{k=1}^\infty$ such that $u^{(k_i,k_i)}$ converges to a point $x\in X$ with
\begin{equation}\label{lma1-6}
    \frac{\left\|x^{(k_i,k_i)}|_{[1,k_i]\times[1,k_i]}\right\|_a}{k_i\cdot k_i}\geq \mbox{Fr}_a(X)~(\forall i\geq 1).
\end{equation}
Then, by (\ref{eqM}) and (\ref{lma1-6}),we have
\begin{equation*}
    \mbox{Fr}_a(X)\leq\frac{\left\|x^{(k_i,k_i)}|_{[1,k_i]\times[1,k_i]}\right\|_a}{k_i\cdot k_i}\leq \frac{M_{k_i,k_i}^a(X)}{k_i\cdot k_i}~(\forall i\geq 1).
\end{equation*}
Taking $i\to\infty$, since $\mbox{Fr}_a(X)=\lim_{k_1,k_2\to\infty}\frac{M_{k_1,k_2}^a(X)}{k_1\cdot k_2}$, we have
\begin{equation*}
    \lim_{i\to\infty}\frac{\left\|x^{(k_i,k_i)}|_{[1,k_i]\times[1,k_i]}\right\|_a}{k_i\cdot k_i}=\mbox{Fr}_a(X).
\end{equation*}
This implies that
\begin{equation*}
    \overline{d}(\chi_a(x))=\limsup_{k_1,k_2\to\infty} \frac{\left\|x^{(k_1,k_2)}|_{[1,k_1]\times[1,k_2]}\right\|_a}{k_1\cdot k_2}\geq \mbox{Fr}_a(X).
\end{equation*}
Note that $\overline{d}(\chi_a(x))\leq \mbox{Fr}_a(X)$. We obtain that $\overline{d}(\chi_a(x))=\mbox{Fr}_a(X)$. 
\item[\bf (3)] The proofs of (3), (4), (5) and (6) are similar to the proofs of Lemma 1, 2, 3 and 4 in \cite{falniowski2015two}, we omit them here.  
\end{proof}

\begin{proof}[Proof of Theorem \ref{Thm: 3}]
The proof just follows the proof of Theorem \ref{Thm: 2} with Lemma \ref{lemma1}. Precisely, we assume that $X$ has positive entropy. For $k_1,k_2\geq 1$, let $\mathcal{L}_{k_1,k_2}(X)$ be the collection of characteristic functions of sets of independent for $B_{k_1,k_2}(X)$. We can treat each element of $\mathcal{L}_{k_1,k_2}(X)$ as a binary block; then $\mathcal{L}(X):=\cup_{k_1,k_2=1}^\infty \mathcal{L}_{k_1,k_2}(X)$ is a factorial and prolongable binary language. We denote the $\mathbb{N}^2$ shift it defines by $I_X$. By compactness of the full shift and the fact that a point $\mathcal{A}^{\mathbb{N}^2}$ is in $X_{\mathcal{L}}$ if and only if $x_{\{i_1,...,i_2\}\times\{j_1,...,j_2\}} \in \mathcal{L}$ for all $i_1, i_2,j_1,j_2 \in \mathbb{N}$ with $i_1 < i_2$ and $j_1<j_2$, elements of $I_X$ may be identified with characteristic functions of independence sets of $X$. It follows from (6) of Lemma \ref{lemma1} that $\mbox{Fr}_1(I_X) > 0$. Using (2) of Lemma \ref{lemma1} we may now fix an element of $I_X$ which is a characteristic function of $J\subseteq \mathbb{N}^2$, an independence set for $X$, such that $\overline{d}(J) = \mbox{Fr}_1(I_X) > 0$.

Conversely, if $J$ is an independence set for $X$ such that $\overline{d}(J) >\delta> 0$, then there exists a subset $\{(n_i,m_i)\}_{i=1}^\infty$ of $\mathbb{N}^2$ such that for $i\geq 1$ there are at least $n_i\cdot m_i\cdot\delta$ elements of $J$ in $[1,n_i]\times [1,m_i]$. Therefore,
\begin{equation*}
    \left|B_{n_i,m_i}(X)\right|\geq 2^{n_im_i\delta} ~(\forall i\geq 1).
\end{equation*}
This implies that $h(X)\geq \delta >0$.
\end{proof}

We remark that for the case $d\geq 2$, we only need to fix $d-1$ directions and prove a similar version of Lemma \ref{lemma1} (1) and (2).
\section{Proof of Theorem \ref{Thm: 1}}\label{sec3}
\begin{proof}[Proof of Theorem \ref{Thm: 1}]
  \item[\bf (1)] Let $X\subseteq [0,m-1]^\mathbb{N}$ be a hereditary shift. Note that $0^n u\in B(X)$ for all $u\in B(X)$, where $B(X)$ is the set of all finite blocks of $X$. By \cite[Theorem 1.1]{ban2024topological}, $h(\mathcal{T}_X)$ is always less than or equal to $h(X)$ when $T$ is an unexpandable tree. It remains to show that $h(\mathcal{T}_X)\geq h(X)$ when $T$ is an unexpandable Markov-Cayley tree. Note that an unexpandable Markov-Cayley tree is a Markov-Cayley tree with adjacency matrix $M$ having maximum eigenvalue 1, and the matrix $M$ can be considered to be of the form
    \begin{equation*}
        M=\left[\begin{matrix}
          C_1&\cdots&*\\
          \vdots&\ddots&\vdots\\
          0&\cdots&C_k
        \end{matrix}\right],
    \end{equation*}
    where $C_1,...,C_k$ are irreducible cycles, and * means that the remaining elements of the upper triangle of $M$ are nonnegative. (Fact: an irreducible 0-1 matrix with maximum eigenvalue 1 is similar to a cycle). We may assume that $C_k$ is the only sink (otherwise, $C_k$ is a union of sinks). Since each ray in $M^n$ must find a way to reach a sink after $|M|$ ($|\cdot|$ also denotes the size of a matrix) steps of walks, and since $|M| < \infty$, we can conclude that the number of rays in $M^n$ ending in a sink is at least $|M^{n-|M|}|$. This gives that
    \begin{equation*}
        \frac{\sum_{i=1}^{|M|}\sum_{j=|M|-|C_k|+1}^{|M|}(M^n)_{ij}}{|M^n|}\geq \frac{|M^{n-|M|}|}{|M^n|}.
    \end{equation*}    
    Since $T$ is unexpandable, we have 
    \begin{equation*}
        \lim_{n\to\infty}\frac{|M^{n-|M|}|}{|M^n|}=\lim_{n\to\infty}\frac{|T_{n-|M|+1}|}{|T_{n+1}|}=1.
    \end{equation*}    
    Thus, 
    \begin{equation}\label{eq1}
        \lim_{n\to\infty}\frac{\sum_{i=1}^{|M|}\sum_{j=|M|-|C_k|+1}^{|M|}(M^n)_{ij}}{|M^n|}=1.
    \end{equation}
    
    We decorate 0 on each ray of the Markov-Cayley tree $T$ before it goes into sinks, then the vertices in sinks can be decorated by any blocks in $B(X)$. This implies that the number of blocks on $\Delta_n$ is greater than or equal to the product of the number of blocks on each sink for some length $i$ from 1 to $n+1$. More precisely, let 
\begin{equation}\label{def al}
    a_\ell =\sum_{i=1}^{|M|}\sum_{j=|M|-|C_k|+1}^{|M|}(M^{n-\ell+1})_{ij}-\sum_{i=1}^{|M|}\sum_{j=|M|-|C_k|+1}^{|M|}(M^{n-\ell})_{ij}
\end{equation}
be the number of paths entering sinks at the $n-\ell$ level, we have     
    \begin{equation*}
        |\mathcal{P}(\Delta_n,\mathcal{T}_X)|\geq \prod_{\ell=1}^{n+1} |\mathcal{P}([1,\ell],X)|^{a_\ell}.
    \end{equation*}
Then,
\begin{equation}\label{eq2}
    \begin{aligned}
    \frac{\log |\mathcal{P}(\Delta_n,\mathcal{T}_X)|}{|\Delta_n|}
    \geq &\frac{\sum_{\ell=1}^{n+1} a_\ell \log|\mathcal{P}([1,\ell],X)| }{\sum_{\ell=1}^{n+1} \ell\cdot a_\ell}\frac{\sum_{\ell=1}^{n+1} \ell\cdot a_\ell}{|\Delta_n|}\\
    \geq &\inf_{i\geq 1} \frac{\log|\mathcal{P}([1,i],X)|}{i} \frac{\sum_{\ell=1}^{n+1} \ell\cdot a_\ell}{|\Delta_n|}\\
    = &h(X)\frac{\sum_{\ell=1}^{n+1} \ell\cdot a_\ell}{|\Delta_n|}.
\end{aligned}
\end{equation}

It remains to show that $\frac{\sum_{\ell=1}^{n+1} \ell\cdot a_\ell}{|\Delta_n|}$ tends to 1 as $n\to\infty$. Note that
\begin{align*}
    &\frac{\sum_{\ell=1}^{n+1} \ell\cdot a_\ell}{|\Delta_n|}\\
    =&\frac{\sum_{\ell=1}^{n+1} \ell\left[\sum_{i=1}^{|M|}\sum_{j=|M|-|C_k|+1}^{|M|}(M^{n-\ell+1})_{ij}-\sum_{i=1}^{|M|}\sum_{j=|M|-|C_k|+1}^{|M|}(M^{n-\ell})_{ij}\right]}{1+|M|+|M^2|+\cdots +|M^n|}\\
    =&\frac{\sum_{\ell=1}^{n+1} \sum_{i=1}^{|M|}\sum_{j=|M|-|C_k|+1}^{|M|}(M^{n-\ell+1})_{ij}}{1+|M|+|M^2|+\cdots +|M^n|}.
\end{align*}
By (\ref{eq1}) and taking $n\to\infty$, we have 
\begin{equation}\label{eq 0-4}
    \lim_{n\to\infty}\frac{\sum_{\ell=1}^{n+1} \ell\cdot a_\ell}{|\Delta_n|}=1.
\end{equation}
By (\ref{eq2}), $h(\mathcal{T}_X)\geq h(X)$. Thus, $h(\mathcal{T}_X)=h(X)$.

\item[\bf (2)]
It is quickly obtained that if $\mathcal{T}_X$ has independence set with positive density, then $h(\mathcal{T}_X)>0$. Conversely, if $h(\mathcal{T}_X)>0$, then by Theorem \ref{Thm: 1}, we have $h(X)>0$. By Theorem \ref{Thm: 2}, we obtain the independence set $S\subseteq \mathbb{N}$ of $X$ of positive density. Then, we put $S$ on each sink. Then, we claim that the union $U$ of S in sinks is the independence set for $\mathcal{T}_X$ and has positive density. Similar to the above theorem, we decorate 0 before a path into a sink, then the same reasoning above leads to the independence of $U$. It remains to check the positive density of $U$. By Theorem \ref{Thm: 2}, we have that $d(S)$ exists. Then, for any $\epsilon >0$, there exists a positive integer $N$ such that if $n\geq N$ then 
\begin{equation}\label{eq 0-1}
    n(d(S)-\epsilon)<\left|S\cap [1,n]\right| <n\left(d(S)+\epsilon \right).
\end{equation}
By (\ref{def al}), we have that for $n\geq 1$, 
\begin{equation*}
    U\cap \Delta_n=\prod_{\ell=1}^{n+1} \left(S\cap[1,\ell]\right)^{a_{\ell}}.
\end{equation*}
This implies that
\begin{align*}
\sum_{\ell=N}^{n+1} a_\ell \left|S\cap[1,\ell]\right|
     \leq&\left|U\cap \Delta_n\right|=\sum_{\ell=1}^{n+1} a_\ell \left|S\cap[1,\ell]\right|
     \leq \sum_{\ell=1}^{N-1}a_{\ell}\ell+\sum_{\ell=N}^{n+1} a_\ell \left|S\cap[1,\ell]\right|.
\end{align*}
Then, by (\ref{eq 0-1}), we have
\begin{equation*}
    \sum_{\ell=N}^{n+1} a_\ell \ell (d(S)-\epsilon)
     <\left|U\cap \Delta_n\right|< \sum_{\ell=1}^{N-1}a_{\ell}\ell+\sum_{\ell=N}^{n+1} a_\ell \ell (d(S)+\epsilon).
\end{equation*}
Hence, 
\begin{align}\label{eq 0-2}
    \frac{\sum_{\ell=N}^{n+1} a_\ell \ell }{|\Delta_n|}(d(S)-\epsilon)
     <\frac{\left|U\cap \Delta_n\right|}{|\Delta_n|}< \frac{\sum_{\ell=1}^{N-1}a_{\ell}\ell}{|\Delta_n|}+\frac{\sum_{\ell=N}^{n+1} a_\ell \ell }{|\Delta_n|}(d(S)+\epsilon).
\end{align}
Note that $T$ is unexpandable, we have that for $n\geq N$ 
\begin{equation}\label{eq 0-3}
    0\leq\lim_{n\to\infty}\frac{\sum_{\ell=1}^{N-1}a_{\ell}\ell}{|\Delta_n|}\leq \lim_{n\to\infty} \frac{|\Delta_n\setminus \Delta_{n-N}|}{|\Delta_n|}=0.
\end{equation}
Combining (\ref{eq 0-2}), (\ref{eq 0-3}) and (\ref{eq 0-4}), we have
\begin{equation*}
    d(S)-\epsilon<\liminf_{n\to\infty} \frac{\left|U\cap \Delta_n\right|}{|\Delta_n|} \leq \limsup_{n\to\infty}\frac{\left|U\cap \Delta_n\right|}{|\Delta_n|} <d(S)+\epsilon.
\end{equation*}
Thus, $d(U)$ exists and is equal to $d(S)$ which is positive. 
\item[\bf (3)] 
 Note that if $S$ is an independence set for $\mathcal{T}_X$, then the restriction of $S$ to each ray is an independence set for $X$. Again, the $h(X)=0$ implies the restriction of $S$ to each ray must be zero density. Now, we claim that $d(S)=0$. Note that for $n\geq 1$, $\Delta_n$ can be decomposed into $|T_{n-\ell+1}|-|T_{n-\ell}|$ many one-dimensional lattices with cardinality $\ell$ for all $1\leq \ell \leq n+1$ where $T_{-1}:=0$. That is,
    \begin{equation}\label{eq3}
        \Delta_n=\bigsqcup_{\ell=1}^{n+1} [1,\ell]^{|T_{n-\ell+1}|-|T_{n-\ell}|},
    \end{equation}
    where $\sqcup$ denotes the disjoint union.
    
    Since the restriction of $S$ on each one-dimensional lattice is independence set for $X$, this gives that the cardinality of $S$ intersecting one-dimensional lattice is less than or equal to the cardinality of maximum independence set for $X$ on this lattice. That is, for $1\leq \ell \leq n+1$,
    \begin{equation}\label{eq4}
        |S\cap [1,\ell]|\leq \max\{|S'|: S'\mbox{ is an independence set for }X \mbox{ on }[1,\ell]\}:=J_\ell.
    \end{equation}
   
    Then, by (\ref{eq3}) and (\ref{eq4}), we have
    \begin{equation}\label{eq 5}
        |S\cap\Delta_n|\leq \sum_{\ell=1}^{n+1} (|T_{n-\ell+1}|-|T_{n-\ell}|) J_\ell.
    \end{equation}
 Note that if $S'$ is an independence set for $X$ on $[1,\ell_1+\ell_2]$, then $S'\cap [1,\ell_1]$ is an independence set for $X$ on $[1,\ell_1]$ and $S'\cap [\ell_1+1,\ell_1+\ell_2]$ is an independence set for $X$ on $[1,\ell_2]$. Then, we have $J_{\ell_1}+J_{\ell_2}\geq J_{\ell_1+\ell_2}$ for all $\ell_1,\ell_2\geq 0$. By Fekete’s Lemma, $\lim_{\ell\to\infty}\frac{J_{\ell}}{\ell}$ exists and is equal to $\inf_{\ell\geq 1}\frac{J_{\ell}}{\ell}$. Then, by $h(X)=0$ and Theorem \ref{Thm: 2}, $\lim_{\ell\to\infty}\frac{J_{\ell}}{\ell}=0$. Thus, fixed an $\epsilon>0$, there exists a $M>0$ such that $\ell\geq M$, $\frac{J_\ell}{\ell}\leq \frac{\epsilon}{2}$. On the other hand, since $T$ is an unexpandable tree, $\lim_{n\to\infty}\frac{|\Delta_n\setminus \Delta_{n-M}|}{|\Delta_n|}=0$, there exists $N>0$ such that $n\geq N$, $\frac{|\Delta_n\setminus \Delta_{n-M}|}{|\Delta_n|}\leq \frac{\epsilon}{2}$. Thus, combining (\ref{eq 5}), for $n> \max\{N,M\}$, 
 \begin{align*}
     0\leq \frac{|S\cap \Delta_n|}{|\Delta_n|}\leq &\frac{\sum_{\ell=1}^{n+1} (|T_{n-\ell+1}|-|T_{n-\ell}|) J_\ell}{|\Delta_n|}\\
     = &\frac{\sum_{\ell=M+1}^{n+1} (|T_{n-\ell+1}|-|T_{n-\ell}|) J_\ell}{\sum_{\ell=1}^{n+1} (|T_{n-\ell+1}|-|T_{n-\ell}|)\ell}+\frac{\sum_{\ell=1}^{M} (|T_{n-\ell+1}|-|T_{n-\ell}|) J_\ell}{|\Delta_n|} \\ 
     \leq&\frac{\sum_{\ell=M+1}^{n+1} (|T_{n-\ell+1}|-|T_{n-\ell}|) J_\ell}{\sum_{\ell=M+1}^{n+1} (|T_{n-\ell+1}|-|T_{n-\ell}|)\ell}+\frac{\sum_{\ell=1}^{M} (|T_{n-\ell+1}|-|T_{n-\ell}|)\ell}{|\Delta_n|} \\ 
     \leq&\frac{\epsilon}{2}+\frac{|\Delta_n\setminus \Delta_{n-M}|}{|\Delta_n|} \\ 
     \leq &\frac{\epsilon}{2}+\frac{\epsilon}{2}=\epsilon.
 \end{align*}
    Hence, $d(S)=0$.
\end{proof}
\section{Proofs of Theorems \ref{Thm: 4}, \ref{Thm: 5} and \ref{Thm: 6}}\label{sec4}

\begin{proof}[Proof of Theorem \ref{Thm: 4}]
        Let $A=\left[\begin{matrix}
        1&1\\
        0&1
    \end{matrix}\right]$. Consider $\mathcal{T}_A$ on $T$. Since for any $n\geq 1$, if the vertices in the $\Delta_{n-1}$ are decorating $0$, then the vertices in $T_n$ will be arbitrarily decorated from $[0,1]$. Thus, the number of blocks on $\Delta_n$ is greater than or equal to $2^{|Tn|}$. Then,
    \begin{align*}
        h(\mathcal{T}_A):=\limsup_{n\to\infty}\frac{\log|\mathcal{P}(\Delta_n,\mathcal{T}_A)|}{|\Delta_n|}
        \geq \limsup_{n\to\infty}\frac{\log 2^{|T_n|}}{\sum_{i=0}^{n}|T_i|}.
    \end{align*}
    Since $\lim_{n\to\infty}\frac{|T_{n+1}|}{|T_{n}|}=\gamma_T>1$, we have
    \begin{equation*}
        h(\mathcal{T}_A)\geq \frac{\log 2}{\sum_{i=0}^\infty \frac{1}{\left(\gamma_T\right)^i}}=\frac{\gamma_T}{\gamma_T -1}\log 2 >0.
    \end{equation*}

    Now, we claim that $\mathcal{T}_A$ has no independence set with positive density. Arguing contrapositively, there exists an independence set $S\subseteq T$ of $\mathcal{T}_A$ with positive density, then, $\lim_{n\to\infty}\frac{|S\cap \Delta_n|}{|\Delta_n|}=\delta>0$. Observe that any two distinct vertices in $S$ can not be prefixed by each other. This is because for any $g,gg'\in T$, by the rule of $A$, there is no $t\in \mathcal{T}_A$ which satisfies $t_{g}=1$ and $t_{gg'}=0$. Consequently, the best choice for $S\cap \Delta_n$ is that all vertices of $S\cap \Delta_n$ are in $T_n$. This implies that
    \begin{equation*}
        |S\cap T_n|=|S\cap \Delta_n|\geq \frac{\delta}{2}|\Delta_n|\geq \frac{\delta}{2} |T_n|\mbox{ for all }n\gg1. 
    \end{equation*}
    Then, by $\frac{M\delta}{2}>1$ for $M=\left\lceil\frac{2}{\delta}\right\rceil<\infty$ (since $\delta>0$) and any two distinct vertices in $S$ can not be prefixed by each other and $|\{i: gg_i\in T\}|\leq d$ for all $g\in T$, we obtain that $S\cap T_n=\emptyset$ for all $n\gg 1$. Then, by $\gamma_T>1$ and $|S\cap T_n|=0~(\forall n\geq m)$, we have
    \begin{equation*}
        d(S)=\lim_{n\to\infty}\frac{\left|S\cap \Delta_n\right|}{\left|\Delta_n\right|}\leq \lim_{n\to\infty}\frac{|\Delta_m|}{|\Delta_n|}=0.
    \end{equation*}
    This contradicts the assumption that $S$ has positive density.
\end{proof}
\begin{proof}[Proof of Theorem \ref{Thm: 5}]
    By \cite[Lemmas 1, 2 and 3]{falniowski2015two} and $h(\mathcal{T})>0$, there exists a constant $\epsilon>0$ such that for each $n\geq 1$, $\mathcal{T}$ is independent on a subset of $\Delta_n$ of cardinality $\epsilon|\Delta_n|$. Since $T$ is an expandable tree ($\gamma_T>1$), then there exists $L>0$ such that $\frac{(\gamma_T)^{L+1}}{\gamma_T-1}<\epsilon$. This gives that there are at least $(\epsilon-\frac{(\gamma_T)^{L+1}}{\gamma_T-1})|\Delta_n|$ many vertices in $\Delta_n\setminus \Delta_{n-L}$ that are independent for $\mathcal{T}$. Note that $\epsilon-\frac{(\gamma_T)^{L+1}}{\gamma_T-1}>0$. Conversely, if $\{S_n\}_{n=\ell}^\infty$ is a sequence of independence sets for $\mathcal{T}$ on $\Delta_n\setminus \Delta_{n-\ell}$ with $\limsup_{n\to\infty}\frac{|S_n\cap \Delta_n|}{|\Delta_n|}=\delta>0$, then there exists a subsequence $\{n_i\}_{i=1}^\infty$ of $\mathbb{N}$ such that 
    \begin{equation*}
        |\mathcal{P}(\Delta_{n_i},\mathcal{T})|\geq 2^{\left|S_{n_i}\right|}~(\forall i\geq 1).
    \end{equation*}
    Thus,
    \begin{align*}
          h(\mathcal{T}):=\limsup_{n\to\infty}\frac{\log|\mathcal{P}(\Delta_n,\mathcal{T})|}{|\Delta_n|}
        \geq \limsup_{i\to\infty}\frac{\log 2^{\left|S_{n_i}\right|}}{|\Delta_{n_i}|}\geq \delta\log2 >0.
    \end{align*}

 A similar argument holds for more numerous alphabets, we omit providing it here.
\end{proof}

\begin{proof}[Proof of Theorem \ref{Thm: 6}]
 If $h(\mathcal{T})>0$, then by a similar argument of \cite[Lemma 2.3]{ban2024topological}, we have $0<h(\mathcal{T})\leq h^{(S)}(\mathcal{T})$. Conversely, if $h^{(S)}(\mathcal{T})>0$, we claim that $\mathcal{T}$ has the boundary independence property, and then, by the equivalence between (1) and (2) in Theorem \ref{Thm: 5}, we have $0<h(\mathcal{T})$. Note that there exist a subsequence $\{n_i\}_{i=1}^\infty$ of $\mathbb{N}$ and $\delta>0$ such that 
   \begin{equation*}
       \left|B_{n_i}^{(S)}(\mathcal{T})\right|\geq 2^{\left|T_{n_i}\right|\delta\log_2 e} ~(\forall i\geq 1).
   \end{equation*}
   Then, by \cite[Lemmas 1 and 2]{falniowski2015two} with an $\epsilon >0$ such that
   \begin{equation*}
       \delta\log_2 e=H(\epsilon):=-\epsilon\log \epsilon -(1-\epsilon)\log (1-\epsilon),
   \end{equation*}
   we have that there exists a subset $I_{n_i}$ of $T_{n_i}$ independent for $\mathcal{T}$ with $\left|I_{n_i}\right|=\left\lfloor \epsilon \left|T_{n_i}\right| \right\rfloor.$ This implies that $\mathcal{T}$ has the boundary independence property.
\end{proof}

\section{Acknowledgements} We sincerely thank Prof. Jian Li for valuable feedback and insightful suggestions on the initial draft of this article. 

Ban is partially supported by the National Science and Technology Council, ROC (Contract No NSTC 111-2115-M-004-005-MY3) and National Center for Theoretical Sciences. Lai is partially supported by the National Science and Technology Council, ROC (Contract NSTC 111-2811-M-004-002-MY2).

{\bf Author Contributions:}\\
Every author has contributed to the project to be included as an author.

{\bf Data availability statement:}\\
No new data were created or analysed in this study.

{\bf Declarations:}\\
Conflict of interest: The authors declare that they have no conflict of interests.

% bibliography ---------------------------------------------------
\bibliographystyle{amsplain}
\bibliography{ban}

\end{document}